\documentclass[a4paper]{article}

\usepackage{amssymb, amsmath, textcomp, amsthm, mathtools}
\usepackage{bbm,dsfont}

\numberwithin{equation}{section}

\title{Hilbert transforms along Lipschitz direction fields: A lacunary model}
\author{Shaoming Guo, Christoph Thiele}
\date{}

\def\R{\mathbb{R}}
\def\N{\mathbb{N}}

\def\Z{\mathbb{Z}}

\def\beq{\begin{equation}}
\def\endeq{\end{equation}}

\theoremstyle{plain}
\newtheorem{thm}{Theorem}[section]

\newtheorem{lem}[thm]{Lemma}

\newtheorem*{conj*}{Lipschitz Differentiation Conjecture}
\newtheorem*{openproblem*}{Open Problem}

\begin{document}
\maketitle

\begin{abstract}

We prove bounds for the truncated directional Hilbert transform in $L^p(\R^2)$ for any $1<p<\infty$ under a combination of a Lipschitz assumption and a 
lacunarity assumption. It is known that a lacunarity assumption alone is 
not sufficient to yield boundedness for $p=2$, and it is a major question 
in the 
field whether a Lipschitz assumption alone suffices, at least for some $p$.

\end{abstract}

\let\thefootnote\relax\footnote{Date: \date{\today}; MSC classes:  42B20, 42B25.}



\section{Introduction}

The directional Hilbert transform and directional maximal operator in the plane
have been recurrently studied in the literature, see \cite{Ba3}, \cite{Ba2}, \cite{BT}, \cite{Guo1}, \cite{Ka}, \cite{Katz}, \cite{LL1}, \cite{LL2}, \cite{NSW} and the references therein. A prominent question concerns suitable 
assumptions on the direction field, under which these operators are bounded in some $L^p(\R^2)$.

To be specific, define directional operators in the following form, parametrized 
by a measurable function $u:\R^2\to (0,1]$, and originally defined for functions $f$ in the Schwartz class ${\mathcal{S}(\R^2)}$:
\beq
H_u f(x,y)=p.v. \int_{-1}^1 f(x-t,y-u(x,y)t)\frac {dt} t\ ,
\endeq
\beq
M_u f(x,y)=\sup_{\epsilon<1} \frac{1}{2\epsilon}\int_{-\epsilon}^\epsilon |f(x-t,y-u(x,y)t)| {dt}\ . 
\endeq

One natural assumption is that $u$ is Lipschitz with 
sufficiently small Lipschitz constant. One of the main 
open questions in the area is whether $H_u$ and $M_u$ are bounded in 
$L^2(\R^2)$ in this case. No $L^p(\R^2)$ bounds other than the trivial 
$L^\infty(\R^2)$ bound for $M_u$ are known. An extensive 
discussion of this pair of conjectures appears in the work of 
Lacey and Li \cite{LL1}, \cite{LL2}.

Another natural assumption is that $u$ takes values
in a  lacunary set. With such a mere assumption on the range of $u$,
the truncation of the integral in $H_u$ 
to $[-1,1]$ and the constraint $\epsilon<1$ in $M_u$ are ineffective as they can be transformed by scaling. The following theorem addresses $M_u$ and is a special case of a theorem proven by Nagel, Stein and 
Wainger \cite{NSW}. For simplicity we shall restrict attention to a particular lacunary set.
\begin{thm}[\cite{NSW}]\label{nsw}
For all $1<p\le \infty$ there is a constant $C_p$ such that  
for all measurable functions $u$ on $\R^2$ with values in the set $\{2^{-j}, j\in \N\}$
we have
\beq
\|\sup_{\epsilon} \frac{1}{2\epsilon} \int_{-\epsilon}^\epsilon  |f(x-t,y-u(x,y)t)| {dt}\|_p\le C_p\|f\|_p. 
\endeq
\end{thm}
It is known that a suitable generalization of lacunarity is the precise 
assumption on the range of $u$ to make $M_u$ bounded, see the work of Katz \cite{Katz} 
and Bateman \cite{Ba3} . In contrast, Karagulyan \cite{Ka} proved that $H_u$ is not bounded 
for a lacunary set of directions from $L^2(\R^2)$ to $L^{2, \infty}(\R^2)$.\\

The purpose of this paper is to look at $H_u$ for a combination of 
a Lipschitz assumption and a lacunarity assumption. Let $[x]$ denote the 
largest integer less than $x$. 

\begin{thm}\label{main}
Let $u:\R^2\to (0,1]$ be Lipschitz with $\|u\|_{Lip}\le 1$ and let $v:\R^2\to (0,1]$ 
be defined by
\beq
\log_2 v =[\log_2 u]\ .
\endeq
Then for every Schwartz
function $f$ and every $p\in (1, \infty)$ we have 
\beq\label{0312ee1.1}
\|H_v f\|_p \le C_p \|f\|_p.
\endeq
Here $C_p$ is a constant depending only on $p$.
\end{thm}

Note that the range of $v$ is contained in the set of integer powers of two,
a lacunary set. The function $v$ itself cannot be Lipschitz, unless it
is constant.\\

The above theorem has a dyadic model that we now present. 
For two points $p_1$ and $p_2$ in the unit line segment $(0,1]$ (or the
unit square $(0,1]^2$) we define the dyadic distance to be the length of the smallest dyadic interval $(a,b]$ (or the side length of the smallest dyadic square) that contains both points. For a dyadic interval $I \subset (0,1]$, denote by $h_I$ its $L^2$ normalized Haar function. For a map $v$ from $(0,1]^2$ to $(0,1]$ define the operator 
\beq\label{0412ee1.2}
H_{v,D}f(x,y)= 
\sum_{|J|/|I|\le v(x, y)} \langle f, h_I\otimes h_J\rangle h_I(x)h_J(y)\ ,
\endeq
where the sum runs over all dyadic rectangles $I\times J$ in the unit
square with the stated bound on the eccentricity.
We then have 
\begin{thm}\label{0312theorem1.4}
Let $v:(0,1]^2\to (0,1]$ have Lipschitz constant (with respect to
dyadic metric both on domain and target space) at most $1/2$ 
and assume it is lacunary in the sense it takes values in 
$\{2^{-k}, k\in \N\}$. Then for all $p\in (1, \infty)$, we have 
\beq\label{0312ee1.3}
\left\|H_{v,D} f\right\|_{p} \le C_p \|f\|_{p}.
\endeq
Here $C_p$ is a constant which depends only on $p$.
\end{thm}

We round up the discussion by recalling a third type of assumptions
on the direction field, where a Lipschitz assumption recently surfaced
naturally, namely bi-parameter assumptions. Bateman
\cite{Ba2} and Bateman and the second author \cite{BT} proved 
$L^p$ bounds for $H_u$ under the assumption that $u$ depends
only on one variable, $u(x,y)=u(x,0)$ for all $y$, and $3/2<p<\infty$. 
The first author \cite{Guo1} generalized this to direction fields constant along 
families of Lipschitz curves, highlighting the role of Lipschitz
assumptions in this context.\\

\section{The dyadic model: Proof of Theorem 
\ref{0312theorem1.4}}
 
We write $s=I\times J$ for the dyadic rectangles below and we write
$h_s(x,y)=h_I(x)h_J(y)$.
We write the $p$-th power of the $L^p$ norm of \eqref{0412ee1.2} as an iterated integral:
\beq\label{0412ee2.1}
\int_{0}^1\left[ \int_{0}^1 \left|\sum_J \sum_{I: |J|/|I|\le v(x, y)} \langle f, h_s\rangle h_s(x,y)\ \right|^p dy\right] dx.
\endeq
\begin{lem}\label{0412lemma2.1}
Assume $v$ is as in Theorem \ref{0312theorem1.4}. Let $I\times J$ be a dyadic rectangle in $ (0,1]^2$ and let $x\in I$. 
If for some $y\in J$ we have $|J|/|I|\le v(x,y)$, 
then we have $|J|/|I|\le v(x,y')$ for all $y'\in J$.
\end{lem}
\begin{proof}
Assume to get a contradiction that $v(x,y')< |J|/|I|$ for some $y'\in J$. Since $v(x,y)\ge |J|/|I|$, the smallest dyadic interval containing both $v(x,y')$ and $v(x,y)$ is $(0,v(x,y)]$ and we have $d(v(x,y),v(x,y'))=v(x,y)$. 
This gives a contradiction, since the Lipschitz assumption implies:
\beq
d(v(x,y),v(x,y'))\le d(y,y')/2\le |J|/2\le |J|/(2|I|).
\endeq
This finishes the proof of Lemma \ref{0412lemma2.1}.\end{proof}
Denote by ${\bf I}(x,J)$ 
the set of all dyadic intervals $I\subset (0,1]$ such that
there exists $y\in J$ with $|J|/|I|\le v(x,y)$. Only the intervals in 
${\bf I}(x,J)$ have non-zero contribution to the inner sum of 
\eqref{0412ee2.1}. By Lemma \ref{0412lemma2.1}, the condition
$|J|/|I|\le v(x,y)$ becomes void for intervals in 
${\bf I}(x,J)$ and we may write for \eqref{0412ee2.1}:
\beq\label{0412ee2.8}
\int_{0}^1 \left[\int_{0}^1 \left| \sum_{J}
\left(\sum_{I\in {\bf I}(x,J)}
\langle f, h_s\rangle h_s \right)\right|^p dy\right]dx.
\endeq
We have gained independence of the inner summation constraint in $y$ and may use Littlewood-Paley theory in this variable to estimate the last display by 
\beq
\lesssim \int_{0}^1\left[\int_{0}^1\left( \sum_{J}\left|\sum_{I\in {\bf I}(x,J)} 
\langle f, h_s\rangle h_s \right|^2 \right)^{p/2} dy\right] dx.
\endeq
The set ${\bf I}(x,J)$ is convex in the sense that if it contains two intervals $I$ and $I'$, then it also contains all $I''$ with $I\subset I''\subset I$. Hence we can telescope the Haar sum in the variable $x$ into a difference of two martingale averages and estimate these averages by the maximal operator. Hence the last display is bounded by 
\beq
\lesssim \int_{0}^1\left[\int_{0}^1\left( \sum_{J}\left|M_1( 
\langle f, h_J\rangle_2 h_J) \right|^2 \right)^{p/2} dy\right] dx,
\endeq
where $M_1$ denotes the one-dimensional maximal operator in the $x$-variable
and $\langle f,h_J\rangle_2$ denotes the inner product of $f$ in the second variable with $h_J$. We interchange the order of integration in $x$ and $y$ and apply the 
Fefferman-Stein maximal inequality to estimate the last display by
\beq
\lesssim \int_{0}^1\left[\int_{0}^1\left( \sum_{J}\left|
\langle f, h_J\rangle_2 h_J \right|^2 \right)^{p/2} dx\right] dy\ .
\endeq
Changing the integration order back and applying the Littlewood-Paley 
theory in the second variable again estimate the last display 
by $\lesssim \|f\|_p^p$, as desired.

\section{Proof of Theorem \ref{main}: Initial reductions}\label{2909section2}

We write the Fourier transform of $f$ as
$$
\widehat{f}(\xi,\eta)=\int \int f(x,y) e^{-2\pi i (x\xi+\eta y)}\, dxdy\ .
$$
Our first step is to pass from $H_v$ to a directional Fourier multiplier operator.
The directional Fourier multiplier for a bounded measurable function 
$m:\R\to \R$ 
is defined as
\beq
T[m]f(x,y):= 
\int \int \widehat{f}(\xi,\eta) m (\xi +\eta v(x,y)) e^{2\pi i (x\xi+y\eta)}\, d\xi d\eta \ .
\endeq
Let $\phi$ be a real even Schwartz function with $\phi(0)=1$, whose Fourier 
transform vanishes on $[-1,1]$ and outside $[-2,2]$:
\begin{equation}\label{supportphi}
{\rm supp}(\widehat{\phi})\subset [-2,2]\setminus [-1,1].
\end{equation}
A calculation gives
\beq
T[ \widehat{\phi} * (i\pi {\rm sign})]f(x,y)= 
p.v. \int f(x-t,y-v(x,y)t)\phi(t) \frac {dt} t\ .
\endeq
Comparing with $H_v$, since $t^{-1}(1_{[-1,1]}-\phi(t))$ is bounded and rapidly decaying, we obtain by a standard superposition argument 
$$|(H_v - T[ \widehat{\phi} * (i\pi {\rm sign})])f(x,y)|\le 
C \sup_{\epsilon} \frac{1}{\epsilon} \int_{-\epsilon}^\epsilon  |f(x-t,y-v(x,y)t)| {dt}.$$ 
By Theorem \ref{nsw}, the right hand side is bounded in $L^p$. 
Hence it suffices to prove Theorem \ref{main} with 
$T[ \widehat{\phi} * {\rm sign}]$ in place of $H_v$. In the following, we will denote 
\beq
m=\widehat{\phi} * {\rm sign}.
\endeq

Note that $T[m]$ is well defined for functions whose Fourier transform is rapidly
decaying. Given a Schwartz  function $f$, we split it as a sum of two functions, one whose Fourier transform is supported in $\eta\ge0$ (that is the set
$\{(\xi, \eta): \eta\ge 0\}$), and one whose Fourier transform is supported in 
$\eta\le 0$.
It suffices to prove the bound of Theorem \ref{main} separately for the two functions. By symmetry it suffices to consider the first function, we may thus
without loss of generality assume $\widehat{f}$ is supported in 
$\eta\ge 0$.

Next we split the function as a sum of a function whose Fourier transform 
is supported in 
$|\xi| \le 100$ and one whose Fourier transform is supported 
in $|\xi|\ge 100$.
We consider the two functions separately.
First assume $\widehat{f}$ is supported in 
$|\xi| \le 100$.

Let $\phi$ be an even real Schwartz function as in \eqref{supportphi}. Moreover we assume the normalization
\begin{equation}\label{normalizephi}
\int_0^{\infty} \widehat{\phi}(t)\frac{dt}t=1.
\end{equation}
Set 
\begin{equation}\label{scalephi}
\phi_t(x):=t\phi(tx)
\end{equation}
and define the Littlewood-Paley operator in the second variable 
\beq
P_{t}f(x,y)=\int f(x,y-z) \phi_t (z)\, dz .
\endeq
We have by the classical Calder\`on reproducing formula
$$f=\int_0^\infty P_{t} f \frac {dt}{t}.$$
Hence
$$T[m] f(x,y)=\int_0^\infty T_m[P_{t} f](x,y) \frac {dt}{t}$$
\begin{equation}\label{verticallp}
=
\int_{0}^{1000/v(x,y)} T_m[P_{t} f](x,y) \frac {dt}{t}+
\int_{1000/v(x,y)}^\infty P_t f (x,y)\frac {dt}{t} .
\end{equation}
Here we have dropped the operator $T[m]$ in the second integral since for
$t>1000/v(x,y)$ the support of $\widehat{P_tf}$ is contained in
$|\xi| \le 100$ and $\eta\ge 1000/v(x,y)$, where we have
$m(\xi+\eta v(x,y))=1$.

Note that there is a Schwartz function $\varphi(\cdot)$ 
which coincides with 
\begin{equation}\label{teles}
\int_0^1 \phi_t(\cdot) \, \frac{dt}t
\end{equation} outside the origin. Moreover, we denote $\varphi_t(\cdot)=t\varphi(t\cdot)$.
We estimate the second term in $\eqref{verticallp}$ by
$$|\int_{1000/v(x,y)}^\infty P_t f (x,y)\frac {dt}{t}- f(x,y)|=
|\int_0^{1000/v(x,y)}P_t f (x,y) \frac{dt} t|$$
$$=|1000/v(x,y) \int  \varphi_{1000/v(x,y)}(z) f(x,y-z)\, dz| 
\lesssim M_2 f(x,y) ,$$
where $M_2$ is the maximal operator in the second variable,
$$M_2 f(x,y)=\sup_{\epsilon>0}\frac 1{2\epsilon} 
\int_{-\epsilon}^{\epsilon}|f(x,y-z)|\, dz .$$
The classical bound on $M_2$ in $L^p$ takes care of the second term in
$\eqref{verticallp}$.

To estimate the first term in $\eqref{verticallp}$, we note
$$\int_0^{1000/v(x,y)} T[m] P_t f (x,y) \frac{dt}t$$
$$=
\int \int \widehat{f}(\xi,\eta) \widehat{\varphi}(\eta v(x,y)/1000)m (\xi +\eta v(x,y)) e^{2\pi i (x\xi+y\eta)}\, d\xi d\eta \ .$$
Note that 
$\widehat{f}(\xi,\eta) \widehat{\varphi}(\eta v(x,y)/1000)$ is nonzero only if 
$|\xi|\le 100$ and $ \eta v(x,y)\le 2000 $.
Moreover, 
we have for $n\ge 0$ the symbol estimates 
\beq\begin{split}
& |\partial_\xi ^n (\widehat{\varphi}(\eta v(x,y)/1000)m(\xi+\eta v(x,y)))|\le C_n\ ,\\
& |\partial_\eta ^n(\widehat{\varphi}(\eta v(x,y)/1000) m(\xi+\eta v(x,y))|\le C_n v(x,y)^{n}\ .
\end{split}\endeq
This means that if we dilate the multiplier 
$\widehat{\varphi}(\eta v(x,y)/100)m(\xi+\eta v(x,y))$ by a factor $v(x,y)$
in direction $\eta$ we obtain a ($(x,y)$-dependent)  multiplier with uniform symbol estimates, which therefore can be controlled pointwise by the Hardy-Littlewood maximal
function. The undilated multiplier can then be estimated by the strong maximal function. Hence we obtain for the second term in \eqref{verticallp} the estimate
\beq
\int_0^{1000/v(x,y)}T[m]P_tf(x,y)\frac{dt}t\lesssim M_1 M_2 f(x,y),
\endeq
where $M_1$ and $M_2$ are the maximal operators in the first and 
second variable.

This concludes the case that $\widehat{f}$ is supported in
$|\xi| \le 100$. Henceforth we assume that
$\widehat{f}$ is supported in $|\xi| \ge 100$. 
We note that for $\xi>100$ and $\eta>0$ we have
$m(\xi+\eta v(x,y))=1$ and hence $T[m]f=f$
if $f$ is supported in $\xi>100$. We may thus assume
$f$ is supported in $\xi<-100$.\\

In what follows, the truncation of the directional Hilbert transform
turns out unnecessary and we shall for simplicity remove it as follows.
By the fundamental theorem of calculus we write
\beq
m(\xi)=
-1+\int_{-10}^{10} m'(\tau)1_{(\tau,\infty)}(\xi)\, d\tau,
\endeq
where we have used that $m'$ vanishes outside $[-10,10]$.
Since $T[-1]$ is clearly bounded, it suffices to prove
bounds on $T[1_{(\tau,\infty)}]$ for any $|\tau|\le 10$.

Define
\beq
H_{\tau, j}f(x,y):= 
\int \int \widehat{f}(\xi,\eta) 1_{(\tau,\infty)} (\xi +\eta 2^{-j}) e^{2\pi i (x\xi+y\eta)}\, d\xi d\eta \ .
\endeq
We need the Cordoba-Fefferman \cite{CF2} inequality in the form of Theorem 6.1
in \cite{DS}. It states that for any $1<p<\infty$ and any collection of functions $f_j$
in $L^p(\R^2)$ we have  
\beq
\|(\sum_j|H_{0,j}f_j|^2)^{1/2}\|_p\lesssim \|(\sum_j|f_j|^2)^{1/2}\|_p .
\endeq
This is proved in \cite{DS} using the maximal function bound of Theorem 
\eqref{nsw}. Applying the above estimate with the modulated functions
$$f_j(x,y) e^{2\pi i \tau x},$$ we obtain more generally 
\begin{equation}\label{cfineq}
\|(\sum_j|H_{\tau ,j}f_j|^2)^{1/2}\|_p\lesssim \|(\sum_j|f_j|^2)^{1/2}\|_p .
\end{equation}

Let $\phi_s$ be as in \eqref{supportphi}, \eqref{normalizephi}, 
\eqref{scalephi} and define the Littlewood-Paley operator in the first variable
\beq
P_{s,1}g(x,y):=\int g(x-z,y) \phi_s (z)\, dz .
\endeq
Since $\widehat{f}$ is supported in $\xi<-100$, we have the Calder\`on reproducing
formula in the form 
\beq
f= \int_{10}^\infty P_{s,1} f\,\frac{ds}s .
\endeq
Now let $\psi$ be an even real Schwartz function supported in $[-1,1]$ with $\widehat{\psi}(0)=0$
and 
$$\int_0^\infty (\widehat{\psi}(t))^2\frac {dt}t=1,$$
and define 
\beq
P_{t,2}g(x,y):=\int g(x,y-z) \psi_t (z)\, dz .
\endeq
Then by the reproducing formula again,
\beq
f= \int_{10}^\infty \int_0^\infty P_{t,2}^2P_{s,1} f\,\frac{dt}t\frac{ds}s.
\endeq
For $j\in \Z$ let $E_j$ be the set of $(x,y)$ such that $v(x,y)=2^{-j}$.
Then
$$T[1_{(\tau,¸\infty)}] f
=\sum_{j\ge 0} 1_{E_j} H_{\tau,j}  f
$$
\begin{equation}\label{dsum}
=\sum_{d\in \Z} \sum_{j\ge 0}  1_{E_j} 
\int_{10}^\infty \int_{2^{j+d}s}^{2^{j+d+1}s}  H_{\tau,j}  P_{t,2}^2P_{s,1} f\,\frac{dt}t\frac{ds}s.
\end{equation}
We consider a summand in \eqref{dsum} for fixed $d$. Setting
\beq
f_j= \int_{10}^\infty \int_{2^{d+j}s}^{2^{d+j+1}s}   P_{t,2}^2P_{s,1} f\,\frac{dt}t\frac{ds}s,
\endeq
we recognise with the Cordoba-Fefferman inequality \eqref{cfineq}
$$
\left\|\sum_j  1_{E_j} 
\int_{10}^\infty \int_{2^{j+d}s}^{2^{j+d+1}s}  H_{\tau,j}  P_{t,2}^2P_{s,1} f\,\frac{dt}t\frac{ds}s\right\|_p$$
$$
=\left\|\sum_j  1_{E_j} 
 H_{\tau,j}  f_j \right\|_p\le \|(\sum_j |H_{\tau,j}f_j|^2)^{1/2}\|_p\lesssim 
\|(\sum_j |f_j|^2)^{1/2}\|_p.
$$
To estimate the last term, by Khintchine's inequality, we need to estimate
\beq
\sum_j \epsilon_j f_j=\sum_j \epsilon_j
\int_{10}^\infty \int_{2^{d+j}s}^{2^{d+j+1}s}   P_{t,2}^2P_{s,1} f\,\frac{dt}t\frac{ds}s
\endeq
in $L^p$, uniformly in all choices of $|\epsilon_j|\le 1$.
However, the last expression is identified as a Marcinkiewicz multiplier 
applied to $f$, see Chapter IV in \cite{S}.

This estimates each term in \eqref{dsum} uniformly in $d$. 
To obtain summability in $d$, we proceed to refine this estimate,
both for $d\le -5$ and for $d\ge 5$.

Consider first $d\le -5$. 
Let $\varphi$ be a Schwartz function with $\widehat{\varphi}$ supported 
in $[-2,2]$
and constantly equal to one on $[-1,1]$. Define
\beq
A_{t,2}g(x,y)=\int g(x,y-z) \varphi_t (z)\, dz .
\endeq
We claim that for $d\le -5$ the summand in \eqref{dsum} is equal to
\beq
\sum_j  1_{E_j} 
\int_{10}^\infty \int_{2^{j+d}s}^{2^{j+d+1}s}  
H_{\tau,j} (1- A_{2^{-d/2}t,2})  P_{t,2}^2P_{s,1} f
\,\frac{dt}t\frac{ds}s .
\endeq
Namely, for fixed $j$ and $(x,y)\in E_j$ and fixed $s>10$ and
$2^{j+d} s\le t\le 2^{j+d+1}s$ we have 
$$H_{\tau,j} A_{2^{-d/2}t,2}  P_{t,2}^2P_{s,1} f=$$
\begin{equation}\label{multsupp}
\int\int \widehat{P_{t,2}^2f}(\xi,\eta) 1_{(\tau,\infty)}(\xi+\eta 2^{-j}) 
\widehat{\phi}(s^{-1} \xi) \widehat{\varphi}(2^{d/2}t^{-1} \eta)
e^{2\pi i (x\xi+y\eta)} d\xi d\eta.\end{equation}
We claim the last expression vanishes, since the integrand vanishes.
For $\widehat{P_{t,2}^2f}(\xi,\eta)$ not to vanish we need $\xi<-100$. 
For $\widehat{\phi}(s^{-1} \xi)$ not to vanish we need
$s<|\xi|<2s$, and for $\widehat{\varphi}(2^{d/2}t^{-1} \eta)$ not to vanish
we need $\eta\le 2^{1-d/2}t$. Hence
\beq
\xi+\eta 2^{-j}\le -s + 2^{1-d/2-j}t\le -s+2^{1+d/2}s\le -50.
\endeq
Hence $1_{(\tau,\infty)}(\xi+\eta 2^-j)=0$ since $|\tau|\le 10$ and the 
integrand vanishes.

Applying the Cordoba-Fefferman inequality as above and observing 
that, by the Marcinkiewicz multiplier theorem, for any choices
of $|\epsilon_j|\le 1$ we obtain
\beq
\|\sum_j \epsilon_j
\int_{10}^\infty \int_{2^{d+j}s}^{2^{d+j+1}s}  
P_{t,2}P_{s,1} (1-A_{2^{d/2}t,2}) P_{t,2}
f\,\frac{dt}t\frac{ds}s\|\lesssim 2^{-d},
\endeq
that is we obtain the desired decay in $d$.
Here we have used that for $d\le -5$ the portion
$(1-A_{2^{d/2}t,2}) P_{t,2}$ produces very small 
symbol estimates for the Marcinkiewicz multiplier.

To obtain good bounds for $d\ge 5$ we compare with the operator
\begin{equation}\label{soperator}
S f(x,y):= \sum_{d\ge 5}\sum_j 1_{E_j} 
\int_1^\infty \int_{2^{d+j}s}^{2^{d+j+1}s} 
P_{t,2}^2P_{s,1} f\,\frac{dt}t\frac{ds}s ,
\end{equation}
which we will estimate in the next section.

We claim that
$$\sum_j  1_{E_j} 
\left[\int_{10}^\infty \int_{2^{j+d}s}^{2^{j+d+1}s}  H_{\tau,j}  P_{t,2}^2P_{s,1} f\,\frac{dt}t\frac{ds}s-
\int_{10}^\infty \int_{2^{d+j}s}^{2^{d+j+1}s} 
P_{t,2}^2P_{s,1} f\,\frac{dt}t\frac{ds}s\right]$$
$$= \sum_j  1_{E_j} 
\int_1^\infty \int_{2^{j+d}s}^{2^{j+d+1}s}  H_{\tau,j}  A_{2^{-d/2}t,2}P_{t,2}^2
P_{s,1} f\,\frac{dt}t\frac{ds}s$$
The argument is similar as before. The integrand on the right hand side 
takes again the form \eqref{multsupp}. We need that for those $\eta>0$
with $\widehat{\varphi}(2^{d/2}t^{-1}\eta)\neq 1$, that is
$\eta>t2^{-d/2}$, we have that either the integrand vanishes
or $1_{(\tau,\infty)}(\xi+\eta 2^{-j})=1$. But if the integrand does not vanish, 
we have
\beq
\xi+\eta 2^{-j}\ge -2s + 2^{-d/2-j}t\ge -2s+2^{d/2}s\ge 50
\endeq
This proves the claim. 

We then obtain similarly decay for Marcinkiewicz multiplier estimates,
using that for $d>5$ the portion $A_{2^{-d/2}t,2}P_{t,2}^2$ produces
small symbol estimates since $\widehat{\psi}(0)=0$. This finishes all steps of initial reductions.

\section{Proof of Theorem \ref{main}: Main argument}

We rewrite the operator $S$ as

\beq
\int_{10}^\infty  \int_{32s/v(x,y)}^\infty P_{t,2}^2 P_{s,1} f(x,y) \frac{dt}t
\frac{ds}s .
\endeq
It suffices to prove for every pair of functions $f$, $g$ with $\|f\|_{p}=1$ and $\|g\|_{p'}=1$ that
$$\int \int \int_{10}^\infty  \int_{32 s/v(x,y)}^\infty P_{t,2}^2 P_{s,1} f(x,y) 
\overline{g(x,y)}\frac{dt}t \frac{ds}s \, dxdy\lesssim 1 .$$
We fix $s$ and $x$ and consider the integrand in these variables:
\beq
 \int  \int_{32 s/v(x,y)}^\infty P_{t,2}^2 P_{s,1} f(x,y) 
\overline{g(x,y)}\frac{dt}t dy . 
\endeq
We write out one of the $P_{t,2}$ convolution operators to obtain
for the last display:
\beq
 \int   \int_{32 s/v(x,y)}^\infty \int P_{t,2} P_{s,1} f(x,y-z)
\psi_t(z)  
\overline{g(x,y)} dz\,  \frac{dt}t  dy .
\endeq
We pass the $z$ integration outside the $t$ integration. Then we compare 
the value $v(x,y)$ occuring in the specification of the integraton domain with the value $v(x,y-z)$. The difference we consider as an
error term
\beq E(x,s):=  \int  \int \int_{32 s/v(x,y)}^{32 s/v(x,y-z)}  P_{t,2} P_{s,1} f(x,y-z)
\psi_t(z)  
\overline{g(x,y)}  \,  \frac{dt}t  dz dy ,
\endeq
which we estimate later. In the remaining integral, with $t$-integration
in the domain $s/v(x,y-z)<t<\infty$, we
use the variable $\tilde{y}=y-z$ in place of $y$ and obtain 
$$ \int  \int  \int_{32 s/v(x,\tilde{y})}^\infty  P_{t,2} P_{s,1} f(x,\tilde{y})
\psi_t(z)  
\overline{g(x,\tilde{y}+z)} \,  \frac{dt}t  d\tilde{y} dz .$$
The $z$ integral is recognised as a Littlewood-Paley operator
acting on $g$. Writing again the $x,s$ integrations and calling
$\tilde{y}$ again $y$ we need to estimate
\beq
 \int \int \int_{10}^\infty  \int_{32 s/v(x,y)}^\infty  P_{t,2} P_{s,1} f(x,y)
\overline{P_{t,2}g(x,y)}\,  \frac{dt}t  \frac{ds}s \, dxdy .
\endeq

We now interchange the $s$ and $t$ integrations and apply Cauchy-Schwarz
and H\"older to estimate the last display by
$$ \int \int \int_{320 /v(x,y)}^\infty  \int_{10}^{t v(x,y)/32 }P_{t,2} P_{s,1} f(x,y)
\overline{P_{t,2}g(x,y)}\,  \frac{ds}s  \frac{dt}t \, dxdy$$
$$ \le \int \int (\int_{0}^\infty | \int_{10}^{tv(x,y)/32}P_{t,2} P_{s,1} f(x,y) 
\frac{ds}s |^2\frac{dt}t)^{1/2}
(\int_0^\infty |\overline{P_{t,2}g(x,y)}|^2\frac {dt}t)^{1/2} 
\, dxdy$$
$$ \le \left\| \left\|  (\int_{0}^\infty | \int_{10}^{tv(x,y)/32}P_{t,2} P_{s,1} f(x,y) 
\frac{ds}s |^2\frac{dt}t)^{1/2} \right\|_{L^p(x)}\right\|_{L^p(y)}.$$
Here we have used the Littlewood Paley square function
estimate for the function $g$ in the second variable.
Telescoping the $s$ integral similarly to \eqref{teles} 
and using the maximal function 
$M_1$ in the first direction, we estimate the last display by
\beq
 \lesssim 
\left\| \left\|  (\int_{0}^\infty M_1(P_{t,2} f)(x,y) 
 )^2\frac{dt}t)^{1/2} \right\|_{L^p(x)}\right\|_{L^p(y)} .
 \endeq
Applying the Fefferman-Stein vector-valued inequality for the maximal
function in the first variable gives the bound
\beq
 \lesssim 
\left\| \left\|  (\int_{0}^\infty P_{t,2} f(x,y) 
\frac{ds}s |^2\frac{dt}t)^{1/2} \right\|_{L^p(x)}\right\|_{L^p(y)} .
\endeq
Commuting the $L^p$ norms and using the Littlewood-Paley square function estimate in the second variable controls the last display by  $\lesssim \|f\|_p=1$.\\

We turn to the estimate for the error term $E$.
We distinguish the case $v(x,y-z)<v(x,y)$ and
$v(x,y-z)>v(x,y)$ , more precisely we write $E=E_<+E_>$ with $E_<(x,s)$
equal to  
$$
\int  \int 1_{s/v(x,y-z)>s/v(x,y)}\int_{32 s/v(x,y)}^{32 s/v(x,y-z)}  P_{t,2} P_{s,1} f(x,y-z)
\psi_t(z)  
\overline{g(x,y)}  \,  \frac{dt}t  dz dy .$$
Define $w:\R^2\to (0,1]$ by 
\beq
\log_2 w(x,y)= \left[1/2+\log_2 u\right],
\endeq
that is $w$ at any point is either equal to $v$ or twice as large.
We claim that $E_<(x,s)$ is equal to  
\begin{equation}\label{domainclaim}
\int  \int \int_{32 s/ w(x,y-z)}^{64 s/w(x,y-z)}  P_{t,2} P_{s,1} f(x,y-z)\times 
\end{equation}
$$\psi_t(z)  
\overline{g(x,y)} 1_{2v(x,y-z)=w(x,y-z)} 1_{v(x,y)=w(x,y)} \,  \frac{dt}t  dz dy .$$

To see the claim, note that the only change between the expressions concerns  the
coding of the domain of integration. Thus we need to show that the domains of integration are equal. Consider a point in the domain of integration
of the defining expression for $E_>$. We thus have $s/v(x,y-z)>s/v(x,y)$. We may also  assume
that $\psi_t(z)\neq 0$, or else the integrand vanishes. Then $|z|\le t^{-1}$ from the support
of $\psi_t$. By the Lipschitz assumption on $u$, 
\beq
|u(x,y)-u(x,y-z)|\le  |z|\le  t^{-1}\le  u(x,y-z)/32 , 
\endeq
where in the last inequality we have used $t\in [32 s/v(x,y), 32 s/v(x,y-z)]$ and $s>1$.
Thus $\log_2 u(x,y)$ and $\log_2 u(x,y-z)$ differ by less than $1/4$. 
Since $v(x,y)>v(x,y-z)$ with a strict inequality, then $u(x,y)$ has to be 
slightly above an integer power of $2$ and $u(x,y-z)$ has to be slightly below
the integer power. Hence we conclude
\beq
2 v(x,y-z)=v(x,y)=w(x,y-z)=w(x,y).
\endeq
This shows that
the essential domain of integration in the defining expression of $E_<$
is contained in the essential domain of integration of \eqref{domainclaim}.
The converse implication is rather straight forward.\\

Setting $\tilde{g}(x,y)= g(x,y)1_{v(x,y)=2w(x,y)}$ we thus obtain with the change of variables $\tilde{y}=y-z$ for $E_<(x,s)$
\beq
 \int  \int \int_{32 s/w(x,\tilde{y})}^{64 s/w(x,\tilde{y})}  P_{t,2} P_{s,1} f(x,\tilde{y})
\psi_t(z)  
\overline{P_{t,2}\tilde{g}(x,\tilde{y})} 1_{v(x,\tilde{y})=w(x,\tilde{y})}  \,  \frac{dt}t  dz d\tilde{y}
\endeq
Integrating in $x$ and $s$ with the $s$ integration inside the $t$ integration
and applying Cauchy-Schwarz in the $t$ integration and calling $\tilde{y}$ again $y$
 we estimate the last display similarly to before by
$$ \le \int \int (\int_{0}^\infty | \int_{32t w(x,y)}^{64t w(x,y) }P_{t,2} P_{s,1} f(x,y) 
\frac{ds}s |^2\frac{dt}t)^{\frac 1 2}
(\int_0^\infty |\overline{P_{t,2}\tilde{g}(x,y)}|^2\frac {dt}t)^{\frac 1 2} 
\, dxdy$$
$$ \lesssim \left\| \left\|  (\int_{0}^\infty | M_1(P_{t,2}  f)(x,y) 
\frac{ds}s |^2\frac{dt}t)^{1/2} \right\|_{L^p(x)}\right\|_{L^p(y)}\lesssim 1 .$$
The finishes the estimate for $E_<$. The estimate for $E_>$ is very similar, hence we leave it out.

\section*{Acknowledgement}
The second author acknowledges support by the Hausdorff Center for Mathematics. We thank Pavel Zorin-Kranich and Olli Saari for pointing out an inaccuracy in Lemma 2.1 of the originally
posted version of the preprint.

Shaoming Guo, Institute of Mathematics, University of Bonn\\
\indent Address: Endenicher Allee 60, 53115, Bonn\\
\indent Current address: 831 E Third St, Bloomington, IN 47405\\
\indent Email: shaoguo@iu.edu\\

Christoph Thiele, Institute of Mathematics, University of Bonn\\
\indent Address: Endenicher Allee 60, 53115, Bonn\\
\indent Email: thiele@math.uni-bonn.de\\


\begin{thebibliography}{20}
\bibitem{Ba3} M. Bateman: Kakeya sets and directional maximal operators in the plane. Duke Math. J. 147 (2009), no. 1, 55-77. 
\bibitem{Ba2} M. Bateman: Single annulus $L^ p$ estimates for Hilbert transforms along vector fields, Rev. Mat. Iberoam. 29 (2013), no. 3, 1021-1069.
\bibitem{BT} M. Bateman and C. Thiele: $ L^ p $ estimates for the Hilbert transforms along a one-variable vector field. Anal. PDE 6 (2013), no. 7, 1577-1600. 
\bibitem{CF1} A. Cordoba and C. Fefferman: A weighted norm inequality for singular integrals. Studia Math. 57 (1976), no. 1, 97-101. 
\bibitem{CF2} A. Cordoba and R. Fefferman: On the equivalence between the boundedness of certain classes of maximal and multiplier operators in Fourier analysis. Proc. Nat. Acad. Sci. U.S.A. 74 (1977), no. 2, 423-425.
\bibitem{DS} C. Demeter and P. Silva: 
Some new light on a few classical results.
Colloq. Math.140 (2015), no. 1, 129--147.
\bibitem{Guo1} S. Guo: Hilbert transform along measurable vector fields constant on Lipschitz curves: $L^2$ boundedness. Anal. PDE 8 (2015), no. 5, 1263-1288.
\bibitem{Ka} G. A. Karagulyan: On unboundedness of maximal operators for directional Hilbert transforms. Proc. Amer. Math. Soc. 135 (2007), no. 10, 3133-3141 (electronic). 
\bibitem{Katz} N. Katz: Maximal operators over arbitrary sets of directions. Duke Math. J. 97 (1999), no. 1, 67-79. 
\bibitem{LL1} M. Lacey and X. Li: Maximal theorems for the directional Hilbert transform on the plane. Trans. Amer. Math. Soc. 358 (2006), no. 9, 4099-4117.
\bibitem{LL2} M. Lacey and X. Li: On a conjecture of E. M. Stein on the Hilbert transform on vector fields.  Mem. Amer. Math. Soc. 205 (2010), no. 965, viii+72 pp. ISBN: 978-0-8218-4540-0 
\bibitem{NSW} A. Nagel, E. Stein and S. Wainger: Differentiation in lacunary directions. Proc. Nat. Acad. Sci. U.S.A. 75 (1978), no. 3, 1060-1062. 
\bibitem{S} E. Stein: Singular integrals and Differentiability Properties of Functions. Princeton Mathematical Series, No. 30, Princeton University Press, Princeton, N.J. (1970)



\end{thebibliography}
\end{document}